\documentclass[openany, amssymb, psamsfonts]{amsart}
\usepackage{mathrsfs,comment}
\usepackage[usenames,dvipsnames]{color}
\usepackage[normalem]{ulem}
\usepackage{url}
\usepackage[all,arc,2cell]{xy}
\UseAllTwocells
\usepackage{enumerate}
\usepackage{graphicx} 
\usepackage{hyperref}  
\hypersetup{%
  bookmarksnumbered=true,%
  bookmarks=true,%
  colorlinks=true,%
  linkcolor=black,%
  citecolor=cyan,%
  filecolor=blue,%
  menucolor=blue,%
  pagecolor=blue,%
  urlcolor=blue,%
  pdfnewwindow=true,%
  pdfstartview=FitBH}

\def\makeautorefname#1#2{\expandafter\def\csname#1autorefname\endcsname{#2}}
\def\equationautorefname~#1\null{(#1)\null}
\makeautorefname{footnote}{footnote}%
\makeautorefname{item}{item}%
\makeautorefname{figure}{Figure}%
\makeautorefname{table}{Table}%
\makeautorefname{part}{Part}%
\makeautorefname{appendix}{Appendix}%
\makeautorefname{chapter}{Chapter}%
\makeautorefname{section}{Section}%
\makeautorefname{subsection}{Section}%
\makeautorefname{subsubsection}{Section}%
\makeautorefname{theorem}{Theorem}%
\makeautorefname{thm}{Theorem}%
\makeautorefname{cor}{Corollary}%
\makeautorefname{lem}{Lemma}%
\makeautorefname{prop}{Proposition}%
\makeautorefname{pro}{Property}
\makeautorefname{conj}{Conjecture}%
\makeautorefname{defn}{Definition}%
\makeautorefname{notn}{Notation}
\makeautorefname{notns}{Notations}
\makeautorefname{rem}{Remark}%
\makeautorefname{quest}{Question}%
\makeautorefname{exmp}{Example}%
\makeautorefname{ax}{Axiom}%
\makeautorefname{claim}{Claim}%
\makeautorefname{ass}{Assumption}%
\makeautorefname{asss}{Assumptions}%
\makeautorefname{con}{Construction}%
\makeautorefname{prob}{Problem}%
\makeautorefname{warn}{Warning}%
\makeautorefname{obs}{Observation}%
\makeautorefname{conv}{Convention}%

\newtheorem{thm}{Theorem}[section]
\newtheorem{cor}{Corollary}[section]
\newtheorem{prop}{Proposition}[section]
\newtheorem{lem}{Lemma}[section]

\newtheorem{conj}{Conjecture}[section]

\theoremstyle{definition}
\newtheorem{defn}{Definition}[section]

\newtheorem{quest}{Question}[section]
\newtheorem{rem}{Remark}[section]

\makeatletter
\let\c@obs=\c@thm
\let\c@cor=\c@thm
\let\c@prop=\c@thm
\let\c@lem=\c@thm
\let\c@prob=\c@thm
\let\c@con=\c@thm
\let\c@conj=\c@thm
\let\c@defn=\c@thm
\let\c@notn=\c@thm
\let\c@notns=\c@thm
\let\c@exmp=\c@thm
\let\c@ax=\c@thm
\let\c@pro=\c@thm
\let\c@ass=\c@thm
\let\c@warn=\c@thm
\let\c@rem=\c@thm
\let\c@sch=\c@thm
\let\c@equation\c@thm
\numberwithin{equation}{section}
\makeatother

\title{Rigidity of Totally Geodesic Hypersurfaces in Negative Curvature}

\author{Ben Lowe}

\date{}

\begin{document}

\begin{abstract}

Let $M$ be a closed hyperbolic manifold containing a  totally geodesic hypersurface  $S$, and let $N$ be a closed Riemannian manifold homotopy equivalent to $M$ with sectional curvature bounded above by $-1$.   Then  it follows from the work of Besson-Courtois-Gallot  that $\pi_1(S)$ can be represented by a hypersurface $S'$ in $N$ with volume less than or equal to that of $S$.  We study the equality case: if $\pi_1(S)$ cannot be represented  by a hypersurface $S'$ in $N$ with volume strictly smaller than that of $S$, then  must $N$ be isometric to $M$?  We show that many such $S$ are rigid in the sense that the answer to this question is  positive.  On the other hand, we construct examples of $S$ for which the answer is negative.  

\end{abstract}

\maketitle

\section{Introduction}

There is a sharp upper bound for the area of a minimal surface in a Riemannian 3-manifold with a negative upper bound on its sectional curvature.   Namely, if $(M,g)$ is a Riemannian 3-manifold containing a minimal surface $\Sigma$, and if the sectional curvature $K_g$ satisfies $K_g \leq -1$, then 
\begin{equation} \label{negcurv} 
\text{Area}_g(\Sigma) \leq 2\pi |\chi (\Sigma)|.  
\end{equation} 
It is simple to check this inequality using the Gauss equation and the Gauss Bonnet formula, but the question of what happens in the case of equality is more interesting.  Equality in (\ref{negcurv}) implies that $\Sigma$ is totally geodesic and hyperbolic (constant curvature $-1$) in its induced metric.   The straightforward way for this to happen is if $g$ itself is hyperbolic, but is this the only way?

 V. Lima showed that the answer to this question is negative by constructing non-hyperbolic examples of  metrics on $\Sigma \times \mathbb{R}$ with sectional curvature at most $-1$ for which equality in (\ref{negcurv}) is attained \cite{l19}.  In section \ref{counterexample} we construct closed non-hyperbolic Riemannian 3-manifolds with sectional curvature at most $-1$ that contain totally geodesic hyperbolic  surfaces.  These give examples of metrics for which rigidity in the case of equality in the area inequality (\ref{negcurv}) fails, even in the case when the ambient manifold is closed.    The main goal of this paper is to give conditions under which equality in (\ref{negcurv}) forces the ambient metric to be hyperbolic.  

In fact, our arguments will work in any dimension.   For a closed hyperbolic $n$-manifold $\Sigma$ and a closed $n+1$-manifold $N$ with sectional curvature less than or equal to $-1$ the following holds.  Suppose we have a map  $F: \Sigma \rightarrow N$ so that the induced map $\pi_1(\Sigma) \rightarrow \pi_1(N)$ is injective.  Assume also that there is a closed hyperbolic manifold homotopy equivalent to $N$ via a homotopy equivalence that sends $F(\Sigma)$ to the homotopy class of a totally geodesic hypersurface. Then as we explain in Section \ref{bcgexplanation},  the next theorem is a consequence of the work of Besson-Courtois-Gallot.

\begin{thm}  \label{bcg} 
There exists a smooth map $\overline{F}$ homotopic to $F$ so that the $n$-volume of the image $\overline{F}(\Sigma)$ of $\Sigma$ in $N$ is smaller than the volume of $\Sigma$ in its hyperbolic metric.  The map $\overline{F}$ has the property that in the case that the volume of $\overline{F}(\Sigma)$ is equal to the volume of $\Sigma$ in its hyperbolic metric, then:
\begin{enumerate} 

 \item The image $\overline{F}(\Sigma)$ of $\Sigma$ under $\overline{F}$ is a totally geodesic hypersurface in $N$ with induced metric the hyperbolic metric on $\Sigma$.  
 
 \item The volume of any immersed hypersurface in $N$ homotopic to and distinct from $\overline{F}(\Sigma)$ up to reparametrization is greater than that of $\overline{F}(\Sigma)$.  
 
 \end{enumerate} 
\end{thm} 
\begin{rem} 
If $N$ is three dimensional, the previous theorem is essentially a consequence of the fact that $\overline{\Sigma}$ is homotopic to a least area minimal immersion \cite{sy79} and (\ref{negcurv}.)
\end{rem} 
 Theorem \ref{bcg}  leaves open the question of whether equality forces not just the image of some map homotopic to $F$,  but also the ambient metric $N$, to be hyperbolic. We will show that the answer to this question depends on how the image of $\Sigma$ sits inside $N$.

We say that a hypersurface $\Sigma$ in a Riemannian n-manifold $N$ is \textit{well-distributed} in $N$ if every point in the universal cover $\tilde{N}$ of $N$ is contained in an embedded solid hypercube all of whose hyperfaces lie in lifts of $\Sigma$ to $\tilde{N}$.    The following is our first main theorem. 

\begin{thm} \label{main} 
 Let $N$ be a Riemannian manifold with sectional curvature at most $-1$ that contains a totally geodesic hyperbolic hypersurface $\Sigma$.  Then if $\Sigma$ is well-distributed in $N$, $N$ must be hyperbolic.  	
	\end{thm}

  Let $M$ be a closed hyperbolic manifold that contains infinitely many totally geodesic hypersurfaces $\Sigma_k$  (see \cite{mr03}, \cite{m76} for examples.)  Then using Theorem \ref{main}, we can prove the following:   

\begin{thm} \label{main2} 
There exists $K\in \mathbb{N}$ so that the following holds. Fix  $\Sigma_k$ for  $k>K$.  Let $N$ be a Riemannian manifold homotopy equivalent to $M$ with sectional curvature at most $-1$.   Assume that there is a totally geodesic hyperbolic hypersurface in $N$ whose fundamental group injectively includes to a subgroup conjugate to $\pi_1(\Sigma_k)$ in $\pi_1(N)$, where we have used the homotopy equivalence to identify $\pi_1(M)$ and $\pi_1(N)$.  Then $N$ is isometric to  $M$ in its hyperbolic metric.  
	\end{thm} 

\begin{rem} 
	
	Rather than assuming that $\Sigma$ (resp. $\Sigma_k$) was totally geodesic and hyperbolic in the previous two theorems, we could instead have assumed,  in view of Theorem \ref{bcg},  that $\Sigma$ (resp. $\Sigma_k$) was not homotopic in $N$  to an immersed hypersurface with volume smaller than the volume of $\Sigma$ (resp. $\Sigma_k$) in its hyperbolic metric.

\end{rem}

 Although it is unclear how sharp Theorem \ref{main2} is, certainly some assumption on the hypersurface in the model case  is needed by the constuction in Section \ref{counterexample}.   To the author's knowledge, this is the first example of a rigidity phenomenon for minimal hypersurfaces that is sensitive to how the minimal hypersurface is distributed in the ambient space.  We comment that Farrell-Jones proved that if $N$ is a closed manifold homotopy equivalent to a closed hyperbolic manifold $M$ of dimension above four, then $N$ and $M$ must  be homeomorphic \cite{fj89} but not necessarily diffeomorphic \cite{fj90}.  The sectional curvature of such an $N$ can moreover be taken to be pinched between $-1-\epsilon$ and $-1$ for any $\epsilon>0$.  By geometrization, all closed negatively curved 3-manifolds have hyperbolic metrics.

The results of this paper are evidence for the following conjecture.  

\begin{conj} 
	For every $D>0$ and $k\in \mathbb{N}$ , there exists $A=A(D,k)$ so that the following holds.  Let $N$ be a closed Riemannian manifold of dimension $k$, sectional curvature bounded above by $-1$, and diameter at most $D$.  Then if $N$ contains a  totally geodesic hyperbolic hypersurface with volume greater than $A$, $N$ must be isometric to a hyperbolic manifold.  
	\end{conj} 

  One could formulate a similar conjecture for totally geodesic submanifolds $\Sigma$ of higher codimension.  In this case, one would  need to modify the statement by requiring that $\Sigma$ is not contained in some higher dimensional totally geodesic submanifold with small volume, and one would need to allow for the possibility that the universal cover of $N$ is some other rank one symmetric space.  It might also be interesting to drop the compactness assumption on $N$.  


\subsection{Related Work} 

The study of the kinds of rigidity questions considered in this paper goes back to Calabi, who showed the following.    For a Riemannian 2-sphere $S$ with sectional curvature between 0 and 1, every simple closed geodesic must have length at least $2\pi$ \cite{p46}.  When such a  geodesic has length exactly $2\pi$, then $S$ must be isometric to the unit sphere.  (See \cite{theorembycalabi}, where a proof of this result is given and attributed to Calabi.)

 For $\Sigma$ a minimal sphere in a Riemannian manifold $M$ of sectional curvature at most 1, inequality (\ref{negcurv}) is true with the direction of the inequality reversed.  In contrast to the negative curvature case discussed above and the examples of V. Lima, there \textit{is} a general rigidity statement in the positive curvature three dimensional setting, due to Mazet-Rosenberg \cite{mr14}.  Mazet recently proved a version in higher codimension \cite{m21}.    We also mention that Espinar-Rosenberg obtained rigidity statements  for surfaces of constant mean curvature $H\in (0,1)$ that have the largest possible area given that the ambient space has sectional curvature at most $-1$.  

 There are many examples of closed hyperbolic manifolds that contain closed  totally geodesic hypersurfaces (see \cite{mr03} for the three dimensional case, and  \cite{m76} for examples coming from so-called arithmetic hyperbolic manifolds of simplest type.)   An arithmetic hyperbolic manifold that contains one totally geodesic hypersurface must contain infinitely many.  In the other direction, Bader-Fisher-Miller-Stover \cite{bfms21} and Margulis-Mohamaddi (the latter in dimension 3) \cite{mm22} proved that a finite volume hyperbolic manifold with infinitely many totally geodesic hypersurfaces must be arithmetic.  There are examples of closed hyperbolic 3-manifolds that contain no totally geodesic surfaces \cite{mr03}.  It is an open problem whether there are closed hyperbolic manifolds in dimension above three that contain no totally geodesic hypersurfaces.

 Let $M$ be a closed hyperbolizable 3-manifold.   Calegari-Marques-Neves introduced an entropy functional on metrics $g$ defined by  asymptotic counts of minimal surfaces in $(M,g)$, and proved it to be uniquely minimized at the hyperbolic metric over all metrics with sectional curvature at most $-1$ \cite{cmn}.  The way we move information obtained via homogenous dynamics from constant curvature to variable curvature in the proof of Theorem \ref{main2} is inspired by their ideas.   

Results that assume a negative sectional curvature upper bound often have analogues that instead assume a negative Ricci or scalar curvature lower bound.  For example, in our joint work with Neves \cite{ln21} we prove the following analogue of inequality (\ref{negcurv}).  Let $(M,g_{hyp})$ be a closed hyperbolic 3-manifold containing a closed totally geodesic surface $\Sigma$.  Then for any metric $g$ on $M$ with scalar curvature at least $-6$, the area of any surface homotopic to $\Sigma$ must be at least 
\begin{equation} \label{scal} 
	Area_{g_{hyp}}(\Sigma)=-2\pi\chi(\Sigma).
\end{equation}   
For the latter theorem, rigidity in the case of equality holds, in that equality in (\ref{scal}) implies that $g$ must be hyperbolic.  In the setting of our paper, however, whether equality in (\ref{negcurv})  implies that $M$ must be hyperbolic depends on the homotopy class of $\Sigma$.  Thus, perhaps surprisingly, a  sectional curvature upper bound is in this context a weaker constraint than a scalar curvature lower bound.

\subsection{Outline}

This paper has three main parts.  The reader is encouraged to assume that $M$ and $N$  are three dimensional on a first reading.  

In the first part, we prove Theorem \ref{main}. Let $N$ be  a negatively curved manifold containing a totally geodesic hyperbolic hypersurface $\Sigma$ that satisfies the well-distribution condition.  Every point in the universal cover of $N$ is contained in some solid embedded cube $\Box$ whose hyperfaces lie in lifts  $\tilde{\Sigma}$ of $\Sigma$ to the universal cover.   The crucial point for our argument is that there is an isometric embedding $\Phi$ of the boundary of  $\Box$  in $\mathbb{H}^n$.   After constructing $\Phi$, we define a new metric on $\mathbb{H}^n$ by gluing the unbounded connected component of  $\mathbb{H}^n- \Phi(\partial \Box)$ to $\Box$.  Using the Rauch comparison theorem, we are able to argue that this metric must have been isometric to $\mathbb{H}^n$.  We thus conclude that the metric on the interior of $\Box$ has constant curvature $-1$.  Since  the point we chose was arbitrary, this shows that $N$ has constant curvature $-1$.  

In the second  part, we find examples of totally geodesic hypersurfaces $\Sigma$ in closed hyperbolic manifolds $M$ that satisfy a slightly stronger version of the well distribution condition.   Here we rely on work by Mozes-Shah that implies uniform distribution statements for totally geodesic hypersurfaces in closed hyperbolic manifolds.   The strong well distribution condition will imply that for any Riemannian manifold $N$ homotopy equivalent to $M$ with sectional curvature at most $-1$, any totally geodesic hypersurface in $N$ corresponding to $\Sigma$ will satisfy the well distribution condition.  This allows us to apply Theorem \ref{main} to prove Theorem \ref{main2}.  

In the third part, we construct examples of closed Riemannian manifolds $M$ with sectional curvature at most $-1$ that contain totally geodesic hyperbolic hypersurfaces without themselves being hyperbolic.  

\section{Acknowledgements} 
We are grateful to Laurent Mazet  for explaining to us the construction in Proposition \ref{prop1}.  We also thank Antoine Song and Matthew Stover.  The author was supported by NSF grant DMS-2202830.

\section{Background} \label{prelim} 

In this section we collect some facts that we will need in the paper. 

\subsection{} 

Informally speaking, the Rauch comparison theorem states that if one Riemannian manifold is more negatively curved than the other, then its Jacobi fields will grow faster than those of the other.  We will need the following corollary of that theorem (see \cite{ce70}[Cor. 1.35, Remark 1.37].)

\begin{cor} [Corollary of Rauch Comparison Theorem] \label{cofr} 
Let  $(M,g)$ be a Riemannian manifold of dimension $n$ and let  $m\in M $.  Assume the sectional curvature $K_g$ of $(M,g)$ satisfies $K_{g} \leq  -1 $.   Choose a point $m_0 \in \mathbb{H}^n$, and let $I: T M_m \rightarrow T \mathbb{H}^n_{m_0}$ be a linear isometry between the tangent spaces.  Let  $c: [0,1] \rightarrow M$ be a geodesic segment in $(M,g)$ so that $exp_m$ is a nonsingular embedding on  $s \cdot exp_m^{-1} (c(t))$ for  $0\leq s \leq 1$ and $0 \leq t \leq 1$.  Then we have 
\begin{equation} \label{lc} 
L[c] \geq L[exp_{m_0} \circ I \circ exp_m^{-1}(c)].  
\end{equation} 
\noindent Moreover, in the case of equality the image of the map 
\[
(s,t)\mapsto exp_{m}(s \cdot exp_m^{-1}(c(t))),  \hspace{2mm} 0\leq s \leq 1, \hspace{1mm} 0 \leq t \leq 1 
\]
\noindent is a solid totally geodesic triangle, and every 2-plane tangent to the image has constant sectional curvature $-1$.   
\end{cor} 

\subsection{} \label{backgroundnegcurv} 
Let $M$ be a closed hyperbolic n-manifold.  Then $M$ is the quotient of $\mathbb{H}^n$ by a discrete group isomorphic to $\pi_1(M)$ acting properly discontinuously and by isometries.  If $N$ is closed negatively curved Riemannian manifold, then $M$ is homotopy equivalent to $N$ if and only if $\pi_1(M)$ is isomorphic to $\pi_1(N)$.  Assume this is the case, and fix a homotopy equivalence $F:M \rightarrow N$.  We can lift $F$ to a map $\tilde{F} : \tilde{M} \rightarrow \tilde{N}$ that commutes with the actions of $\pi_1(M)$ and $\pi_1(N)$ by deck transformations.  

 The boundary at infinity $\partial_{\infty}(X)$ of a simply connected negatively curved manifold $X$ is the set of geodesic rays in $X$ up to the equivalence relation of remaining at finite distance for all time when parametrized by arc-length.  For any $x\in X$, the exponential map defines a bijection between the unit sphere in the tangent space to $X$ and $\partial_{\infty}(X)$.  $\partial_{\infty}(X)$ is topologized so that this map is a homeomorphism for all $x$.    Every geodesic in $X$ is uniquely determined by its two endpoints in $\partial_{\infty}(X)$, and conversely any two distinct points in $\partial_{\infty}(X)$ determine a geodesic.  

 The fundamental groups $\pi_1(M)$ and $\pi_1(N)$ act by homeomorphisms on $\partial_{\infty}(\tilde{M})$ and $\partial_{\infty}(\tilde{N})$.   The map $\tilde{F}$  induces a homeomorphism $\tilde{F}_{\infty}: \partial_{\infty} \tilde{M} \rightarrow \partial_{\infty} \tilde{N}$ between the boundaries at infinity of $\tilde{M}$ and $\tilde{N}$, that commutes with the actions of $\pi_1(M)$ and $\pi_1(N)$.

 Suppose that $\Sigma$ and $\Sigma_N$ are closed totally geodesic hyperbolic hypersurfaces of $M$ and $N$ so that $F_* \pi_1(\Sigma)$ is conjugate to $\pi_1(\Sigma_N)$ in $\pi_1(N)$.  Let $\tilde{\Sigma}$ be a lift of $\Sigma$ to $\tilde{M}$, and suppose that the intersection
 \[
 G = \cap_{i=1}^{n-k}\gamma_i \tilde{\Sigma}
 \]
is transverse for some $\gamma_1,..,\gamma_{n-k} \in \pi_1(M)$.  Then $G$  is a $k$-dimensional totally geodesic subspace in $\tilde{M}\cong \mathbb{H}^n$.  The equatorial $n-2$-sphere  $\tilde{F}_{\infty}(\gamma_i \tilde{\Sigma})$ in $\partial_{\infty}(N)$ bounds a unique totally geodesic hyperplane that lifts $\Sigma_N$, which by an abuse of notation we call $\gamma_i \tilde{\Sigma}_n$. It will be important for us that
\[
 \cap_{i=1}^{n-k}\gamma_i \tilde{\Sigma}_N
\]
is a totally geodesic k-plane with boundary at infinity $\tilde{F}_{\infty}(\partial_{\infty} G)$.  This implies, for example, that if $\Box$ is an embedded hypercube in $\tilde{M}$ whose hyperfaces are totally geodesic and extend to totally geodesic hyperplanes containing lifts $\gamma \tilde{\Sigma}$ of $\Sigma$ to $\tilde{M}$, then the region $\Box_N$ bounded by  the corresponding lifts of $\Sigma_N$ in $\tilde{N}$   will also be a hypercube.  
\subsection{} \label{msintro} 

We state the result from homogeneous dynamics that we will need. Let $\Sigma_k$ be a sequence of distinct totally geodesic hypersurfaces in a closed hyperbolic n-manifold $M$.  Denote by  $\hat{\Sigma}_k$ their lifts to the Grassmann bundle $Gr_{n-1}(M)$ of unoriented tangent $n-1$ planes to $M$.  The $\hat{\Sigma}_k$ define probability measures $\mu_k$ on $Gr_{n-1}(M)$ by 
\[
\mu_k(f) := \frac{1}{Vol(\hat{\Sigma}_k)}\int_{\hat{\Sigma}_k} f d \hat{\Sigma}_k.  
\]
for $f$ a continuous function on $Gr_{n-1}(M)$, and where $d \hat{\Sigma}_k$ is the volume form for the hyperbolic metric on $\Sigma_k$.  The hyperbolic metric induces a metric on   $Gr_{n-1}(M)$ on which we denote the volume measure, normalized to have unit volume, by $\mu_{Leb}$.  It follows from Ratner's measure classification theorem that any weak-$*$ limit of the $\mu_k$ is equal to a convex combination of measures supported on totally geodesic hypersurfaces and $\mu_{Leb}$.   That we can rule out ergodic components that are supported on totally geodesic hypersurfaces is a consequence of work by Mozes-Shah \cite{ms95}[Theorem 1.1]: 
\begin{thm} \label{ms} 
	The $\mu_k$ weak-$*$ converge to $\mu_{Leb}$.  
	\end{thm} 

In fact, we will only require a weaker statement.  We say that a surface $\Sigma$ is $\epsilon$\textit{-dense} if every tangent plane in $Gr_{n-1}(M)$ is at a distance of at most $\epsilon$ from some tangent plane to $\Sigma$, where the distance is measured in the natural metric on $Gr_{n-1}(M)$.  The corollary below follows from Theorem (\ref{ms}).  

\begin{cor} \label{coms} 
	For every $\epsilon>0$ there is $K$ so that $\Sigma_k$ is $\epsilon$-dense if $k>K$.  
	\end{cor} 

\subsection{} \label{bcgexplanation} 

Finally, we explain how Theorem \ref{bcg} follows from the work of Besson-Courtois-Gallot.  That $\overline{F}$ is homotopic to a smooth map with Jacobian pointwise smaller than 1 is a direct consequence of  \cite{bcg99}[Theorem 1.10], taking the manifold $X$ in the statement of that theorem to be the universal cover $\tilde{N}$ of $N$, and the representation to  correspond to the action of of $\pi_1(\Sigma)\subset \pi_1(N)$ on $\tilde{N}$ by deck transformations.  That the action of $\pi_1(\Sigma)$ is convex cocompact follows from the fact that $N$ is homotopy equivalent to a closed hyperbolic manifold via a homotopy equivalence that sends $F(\Sigma)$ to a totally geodesic hypersurface.    For the equality case, if no immersion homotopic to $\overline{F}$ has smaller volume than the volume of $\Sigma$ in its hyperbolic metric,  then the volume of $\overline{F}(\Sigma)$ is equal to the volume of $\Sigma$ in the hyperbolic metric, and the Jacobian of $\overline{F}(\Sigma)$ must be everywhere equal to 1.  In that case \cite{bcg99}[Theorem 1.2] implies that the differential of $\overline{F}$ is at every point  an isometry onto its image.  The image $\overline{F}(\Sigma)$ is therefore isometric to the hyperbolic metric on $\Sigma$ in its induced metric.  

Since $\overline{F}(\Sigma)$ minimizes volume over all hypersurfaces homotopic to it, we know that $\overline{F}(\Sigma)$ is a minimal hypersurface.  Thus at every point the principal curvatures $\lambda_1,..,\lambda_n$ sum to zero.  If any were nonzero, then we could find $\lambda_i>0$ and $\lambda_j<0$.  The tangent plane to $\overline{F}(\Sigma)$ spanned by the corresponding principal directions would then have sectional curvature strictly less than one by the Gauss equation and the fact that the sectional curvature of $N$ is less than or equal to $-1$, which contradicts the fact that $\overline{F}(\Sigma)$ is hyperbolic in its induced metric.  If there were some other immersed hypersurface in the homotopy class of $\overline{F}(\Sigma)$  with the same volume, then one could argue as in  \cite{bcg08}[Proof of Theorem 1.2] to show that it is equal to $\overline{F}(\Sigma)$.  

\section{Proof of Theorem 1.4}

In this section we give the proof of Theorem \ref{main}.  Again we recommend that the reader assume $N$ is three dimensional on a first reading.

Assume that $\Sigma$ is a closed totally geodesic hypersurface in $N$ that satisfies the well-distribution condition.   This means that every point $p \in \tilde{N}$ is contained in an embedded  solid hypercube $\Box$ whose hyperfaces are contained in lifts of $\Sigma$ to $\tilde{N}$.  	We claim that there is an isometric embedding $\Phi$ of the boundary $\partial \Box$ of $\Box$ in $\mathbb{H}^n$.  

	\subsection{} 
	
	To illustrate some of the ideas, we first give a proof of a similar statement with the boundary of the cube $\partial \Box$ replaced by the boundary $\Delta$ of a solid embedded tetrahedron $\Delta \subset \tilde{N}$ whose faces are contained in lifts of $\Sigma$.  We assume also that $N$ is three dimensional.  We originally tried to work with tetrahedra rather than cubes, but it does not seem possible to follow the approach of  Section \ref{proofmainthm}  using tetrahedra.

		We claim that $\Delta$ isometrically embeds in $\mathbb{H}^3$.   Label the vertices of $\Delta$ as $A,B,C,D$.   Note that the faces of $\Delta$ meet at constant angles, since the faces are totally geodesic-- i.e., for any edge of $\Delta$, the angle that the two faces meeting at that edge make at any point on the edge is the same.  
	
Then since the face $ABC$ is contained in a totally geodesic plane isometric to $\mathbb{H}^2$, we can map $ABC$ isometrically into some copy of $\mathbb{H}^2$ contained in  $\mathbb{H}^3$.  Call this map $\Phi$.

	Note also that the choice of $\Phi$ together with a choice of orientation for $\mathbb{H}^3$ and for the triangle $ABC$ uniquely determines linear isometries between the tangent spaces to $A$, $B$, and $C$ and respectively $\Phi(A)$, $\Phi(B)$, and $\Phi(C)$-- fix such a linear isometry
	\[
	d\Phi_A: T_A(M) \rightarrow T_{\Phi(A)} ( \mathbb{H}^3), 
	\]
	and similarly for $B$ and $C$.  We can then extend $\Phi$ to the line segment $AD$ by sending it to the geodesic ray beginning at $\Phi(A)$ with tangent vector at $\Phi(A)$ equal to the image under $d\Phi_A$ of the tangent vector to $AD$ at $A$, and similarly for $B$ and $C$.    
	
	By the fact that the faces of $\Delta$ meet along constant angles, we know that $\Phi(AD)$ and $\Phi(BD)$ intersect, and that triangle $\Phi(ABD)$ is isometric to triangle $ABD$, and similarly for triangles $ACD$ and $BCD$.  We can thus extend $\Phi$ to an isometry on all of $\Delta$ as desired.

	\subsection{} 
	
	In the cube case, it is possible by a similar but more complicated argument to define an isometric embedding $\Phi$ of $\partial \Box$ in $\mathbb{H}^n$ ``by hand."  Instead, we give a short proof that this is possible using a developing map argument.  
	
	Note that every point $p$ on a hyperface of $\Box$ has a neighborhood in $\partial \Box$ that isometrically embeds in $\mathbb{H}^n$.  If $p$ is contained in the interior of a hyperface then this is immediate.  Otherwise we can first define an isometric embedding $\Phi'$ on the intersection of  a neighborhood of $p$ with a hyperface $H$ containing $p$.  Then the fact that hyperfaces of $\Box$ meet along constant angles, since they are contained in totally geodesic hypersurfaces, allows us to extend $\Phi'$ to a neighborhood of $p$ in $\partial \Box$.

	 To define $\Phi$, we fix a point $p$ in $\partial \Box$, and define $\Phi$ on a neighborhood of $p$ in $\partial \Box$ to be some isometric embedding of that neighborhood in $\mathbb{H}^n$.  Then for any $p' \in \partial \Box$ and any path $\gamma$ in $\partial \Box$ joining $p$ to $p'$, we can use the fact that each point has a neighborhood in $\partial \Box$ that isometrically embeds in $\mathbb{H}^n$ to extend $\Phi$ along the path $\gamma$ and define $\Phi(p')$. Since $\partial \Box$ is simply connected, $\Phi(p')$ is well-defined independent of the choice of $\gamma$.  Therefore the map $\Phi$ is  a well-defined local isometry onto its image from $\partial \Box$ into $\mathbb{H}^n$.  Note that $\Phi$ is an isometry restricted to each hyperface of $\partial \Box$ and that the images of opposite hyperfaces under $\Phi$ are disjoint.   Therefore $\Phi$ is an isometric embedding onto its image.

	\subsection{} 
	
	We define a new Riemannian manifold $H$ as follows.   First, we take the closure of the non-compact connected component of the complement of $\Phi(\Box)$ in $\mathbb{H}^n$ equipped with the hyperbolic metric.  We then ``fill in" the boundary $\Phi(\Box)$ by gluing in the region $R_\Box$ bounded by $\Box$ in $\tilde{N}$ to obtain $H$.  Since $\Phi$ is an isometry, this defines a Riemannian manifold which we denote by  $H$.  Note that apriori the metric tensor of $H$ is only continuous.  
	
	We claim that $H$ must be isometric to $\mathbb{H}^n$.  Our proof is similar to some of the arguments in \cite{sz90} (see also the recent paper \cite{gs22} and the references therein.)  
	
	Every totally geodesic two-dimensional plane $S$ in $\mathbb{H}^n$ that intersects $\Phi(\Box)$ transversely corresponds to a totally geodesic surface with boundary $S'$ in $H$, whose polygonal boundary  is contained in $\Phi(\Box)$.  The surface $S'$ is isometric to $S$ with the compact region bounded by some $m$-sided polygon $P'$ removed.   Note that $P'$ is isometric to a polygon $P$ in $S \cong \mathbb{H}^2$, and label the vertices of $P$ (resp. $P'$) as $v_1,..,v_\ell$ (resp. $v_1',..,v_\ell'$.) 
	
	First consider the case that $P$ is a triangle.  Note that $P$ has the same angles and side-lengths as a geodesic triangle in $\mathbb{H}^2$.  Then since $R_{\Box}$ is geodesically convex and has curvature bounded above by $-1$, the equality case of Corollary \ref{cofr} implies that $P$ bounds a totally geodesic hyperbolic solid triangle in $R_{\Box}$, and  $S'$ thus extends to an embedded hyperbolic plane in $H$.   
	
	To prove the general case, we will show that for any $i,j$  with $1 \leq i,j \leq \ell$ and $ \ell- |i-j|$ or $|i-j|$ less than $k$,  the length of $v_i'v_j'$ is greater than or equal to that of  $v_iv_j$.  Take this as the inductive hypothesis, assume that it holds for $k$, and choose $i$ and $j$ with $|i-j| =k-1$.  Relabeling if necessary we can assume $i=1$ and $j=k$.  
	
	First we claim that for any $m$ such that $1<m<k$, the angle $<v_{m-1}'v_k'v_m'$ is less than or equal to $<v_{m-1} v_k v_m$.  To see this, by the inductive hypothesis $v_{m-1}'v_k'$ and $v_{m}'v_k'$ have length greater than or equal to respectively $v_{m-1}v_k$ and $v_{m}v_k$, and the length of $v_{m-1}'v_m'$ is the same as that of  $v_{m-1}v_m$.  The claim then follows from Corollary \ref{cofr}.  
	
		\begin{figure}
		\includegraphics[width=50mm]{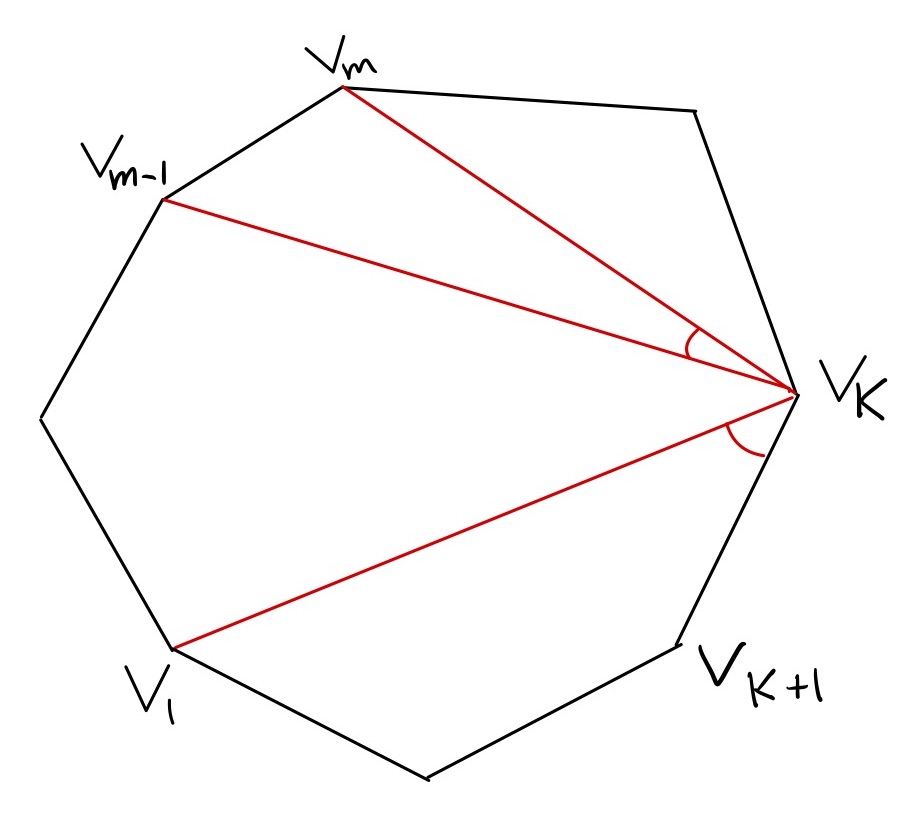}
		\caption{}
	\end{figure}
	
	Second, we claim that angle $>v_1'v_k' v_{k+1}'$ is greater than or equal to  angle $>v_1v_k v_{k+1}$.  Note that  the sum  of $>v_1'v_k' v_{k+1}'$  and 
	\begin{equation} \label{sum23} 
	\sum_{p=2}^{k-1}  >v_{p-1}'v_k'v_p' 
	\end{equation} 
	is equal to $>v_{k-1}'v_k'v_{k+1}'$, which is equal to $>v_{k-1}v_kv_{k+1}$.  Therefore since (\ref{sum23}) is smaller than or equal to  the same sum with the $v_i'$ replaced by $v_i$, this implies that $>v_1'v_k' v_{k+1}'$  is greater than or equal to  $>v_1v_k v_{k+1}$.  Using  Corollary \ref{cofr} we can then conclude that the length of $v_1'v_{k+1}'$ is greater than or equal to the length of $v_1v_{k+1}$.  In the case of equality, we could then use the equality case of Corollary \ref{cofr} to conclude that $v_1'v_k'$ has length equal to $v_1v_k$ and $v_1'v_{k+1}'$ has length equal to $v_1v_{k+1}$, and that $\Delta v_1' v_{k+1}'v_k'$ can be filled in by a totally geodesic hyperbolic triangle.  Continuing in this way we see that we must have had equality at every previous stage of the induction, and we finally get that the length of  $v_1'v_3'$ equals the length of $v_1v_3$, and so $\Delta v_1'v_2'v_3'$ can be filled in by a hyperbolic triangle. Removing $\Delta v_1'v_2'v_3'$ we obtain a new polygon with one fewer side on which we can repeat the same argument.   If we never have equality, then we can conclude the finite induction to get that $v_1'v_\ell'$ has length strictly greater than  $v_1v_\ell$, which is a contradiction since the two have the same length.  We have thus shown that $S'$ extends to an embedded totally geodesic hyperbolic plane.

	\subsection{} 
	
	We explain why this implies that $H$ has constant curvature $-1$.  Take a large metric sphere $B$ in $\mathbb{H}^n$ that contains $\Phi(\Box)$, and let $\mathcal{C}$ be the set of round (but not necessarily great) circles contained in the metric sphere $\partial B$.  There is a natural disk bundle $\mathcal{D}$ over $\mathcal{C}$ whose fiber over $C \in \mathcal{C}$ is equal to the totally geodesic disc in $B$ bounded by $C$.  We define a map $f$ from $\mathcal{D}$ to the Grassmann bundle $Gr_2(\overline{R_{\Box}})$ of unoriented tangent 2-planes to the closure of $R_{\Box}$ as follows.  Here $Gr_2(\overline{R_{\Box}})$  is defined by viewing $\overline{R_{\Box}}$ as a subset of $N$ and pulling back $Gr_2(N)$ to $\overline{R_{\Box}}$.  

 Suppose $d\in \mathcal{D}$ corresponds to a circle $C \in \mathcal{C}$ that bounds a totally geodesic disc $D$ containing $d$. We take the unique totally geodesic tangent plane $S$ that intersects $\partial B$ in  $D$, which by what we have shown above corresponds to a totally geodesic plane $S'$ in $H$.  This $S'$ contains an isometric copy $D'$ of the hyperbolic disc $S \cap B$.  We define $f(d)$ to be the tangent plane $\Pi'$ to the point $d'$ on $D'$ corresponding to $d$ if it is contained in $\overline{R_{\Box}}$, or else the parallel transport of $\Pi'$ along the geodesic joining $d'$ to its nearest point projection to $\partial \overline{R_{\Box}}$.  
 
Note that $f$ is continuous, and $f$ is injective near every point that maps into the interior of $R_{\Box}$.  Since $\mathcal{D}$ and the interior of  $Gr_2(\overline{R_{\Box}})$ both have dimension $3n-4$, the invariance of domain theorem implies that $f$ is a local homeomorphism onto its image near any such point $d$.   Since the image of $f$ is both open and closed, the fact that $Gr_2(\overline{R_{\Box}})$ is connected implies that $f$ is surjective. It follows that $R_{\Box}$ has constant curvature $-1$, which implies that $H$ is isometric to $\mathbb{H}^n$.

\section{Proof of Theorem 1.5} \label{proofmainthm} 

Let $M$ be a closed hyperbolic manifold, and let $N$ be a Riemannian manifold homotopy equivalent to $N$ with sectional curvature at most $-1$.  Fix a homotopy equivalence $F: M \rightarrow N$. Let $\Sigma$ be a  totally geodesic hypersurface in $M$, and assume that $N$ contains a totally geodesic hypersurface $\Sigma_N$ whose fundamental group includes to $\pi_1(\Sigma)$ up to conjugacy, where we've used $F$ to identify $\pi_1(M)$ and $\pi_1(N)$.   To show that $\Sigma_N$ satisfies the well-distribution property, we will actually need $\Sigma$ to satisfy a stronger condition, which we define now. 

\begin{defn} 
	We say that $\Sigma$ satisfies the \textit{strong well-distribution property} if the following holds.  Let $\mathcal{B}_{\Sigma}$ be the set of solid embedded hypercubes in $\mathbb{H}^n=\tilde{M}$ each of whose boundary hyperfaces is contained in a lift   of  $\Sigma$ to the universal cover $\mathbb{H}^n$.  Then for each geodesic $\gamma$ in the universal cover $\tilde{M} \cong \mathbb{H}^n$ of $M$, we can find  hypercubes $\Box_i \in \mathcal{B}_{\Sigma}$, $i \in \mathbb{Z}$,  so that the following holds.  This sequence of hypercubes is then said to \textit{enclose} $\gamma.$

	\begin{enumerate}  
		\item Exactly two hyperfaces $H_{i,top}$ and $H_{i,bottom}$ of $\Box_i$ intersect $\gamma$ 
		\item  The totally geodesic hyperplanes that contain the hyperfaces of $\Box_i$, excluding $H_{i,top}$ and $H_{i,bottom}$, do not intersect $\gamma$.  
		
		\item $\Box_{i-1}$ and $\Box_i$ satisfy the \textit{interlocking property}:  the totally geodesic hyperplane that contains $H_{i-1,top}$ is contained in the region between those of  $H_{i,top}$ and $H_{i,bottom}$.  
		
		\item Each point on $\gamma$ is contained in some $\Box_i$.  
		
	\end{enumerate}

	\end{defn} 

Now let $\Sigma_k$ be an infinite sequence of distinct closed totally geodesic hypersurfaces in $M$.  We claim that for $k$ large enough, $\Sigma_k$ satisfies the strong well-distribution property.  By Corollary \ref{coms} we know that for every $\epsilon>0$ the $\Sigma_k$ are $\epsilon$-dense in $Gr_{n-1}(M)$ for $k$ large enough.

For $\epsilon>0$ sufficiently small, we define the following hypercube $\Box_{model}(\epsilon)$ with totally geodesic faces in $\mathbb{H}^n$.  Take some point $P$ in $\mathbb{H}^n$, and let $E_1,..,E_n$ be an orthonormal basis for the tangent space at that point.  For each $i$, let $H_i^{+}$ and $H_{i}^-$ be the totally geodesic hyperplanes orthogonal to the geodesic through $P$  in the direction of $E_i$, and at signed distance of respectively $\epsilon$ and $-\epsilon$ from $P$.  Let $U_i$ be the connected component of the complement of $H_i^{+} \cup H_{i}^-$ in $\mathbb{H}^n$ containing $P$.  We define $\Box_{model}(\epsilon)$ to be the closure of the intersection $\cap_{i=1}^n U_i$, where we have taken $\epsilon$ small enough that this intersection is a solid embedded hypercube.  

We say that two hypercubes $\Box$ and $\Box'$ with totally geodesic hyperbolic faces are $\delta$-close if there is a labeling of the vertices $V_1,..,V_\ell$ of $\Box$ and the vertices $V_1',..,V_{\ell}'$ of $\Box'$ such that 

\begin{enumerate} 
	\item For all $i,j$ such that $V_i$ and $V_j$ are adjacent the length of $V_iV_j$ and the length of $V_i'V_j'$ differ by less than $\delta$. 
	\item For all $i,j,p$ such that $V_i$ and $V_p$ are adjacent to $V_j$ the angle $>V_iV_jV_p$ differs from the angle $>V_i'V_j'V_p'$ by less than $\delta$.  
	\end{enumerate}

 Denote by $\mathcal{B}_{\Sigma_k}^{\epsilon}$ the set  of hypercubes $\Box$ in $\mathcal{B}_{\Sigma_k}$ that are $\delta$-close to $\Box_{model}(\epsilon)$, for some small $\delta$ be be specified later.  The reason for defining $\mathcal{B}_{\Sigma_k}^{\epsilon}$ is that having lots of elements of $\mathcal{B}_{\Sigma_k}$ of controlled shape will be useful for enclosing $\gamma$.

Fix $\epsilon>0$ small enough that $\Box_{model}(\epsilon)$ is defined.  For every geodesic $\gamma$ in the universal cover, we claim that we can find a collection of solid hypercubes $\Box$ contained in $\mathcal{B}_{\Sigma_k}^{\epsilon}$ that enclose $\gamma$, provided that $k$ was chosen sufficiently large and $\delta$ was chosen sufficiently small at the start.   Let $\mathcal{B}_{\Sigma_k}^{\epsilon}(\gamma)$ be the set of hypercubes $\Box$ in the universal cover so that the following holds: 

\begin{enumerate}
	\item $\Box$ is the lift of some element of $\mathcal{B}_{\Sigma_k}^{\epsilon}$ to the universal cover 
	
	\item $\gamma$ intersects the boundary of $\Box$ in exactly two points $p_1$ and $p_2$ on opposite hyperfaces $H_{top}$ and $H_{bottom}$ of $\Box$  
	
	\item $p_1$ and $p_2$ are at a distance of less than $\delta$ from the centers of $H_{top}$ and $H_{bottom}$ respectively.  Here the centers of $H_{top}$ and $H_{bottom}$ are the endpoints of the geodesic segment of shortest length beginning at $H_{bottom}$ and ending at $H_{top}$.  
	
	\item $\gamma$ makes an angle between $\pi/2 -  \delta $ and  $\pi/2 +  \delta $ with $H_{top}$ and $H_{bottom}$ at respectively $p_1$ and $p_2$ 
	\end{enumerate} 

Here $\delta$ is chosen small enough that for every $\gamma$ and $\Box$ satisfying the above four conditions, $\gamma$ will be disjoint from all of the totally geodesic hyperplanes that contain hyperfaces of $\Box $ different from $H_{top}$ and $H_{bottom}$.  We emphasize that all of the requirements we impose on $\delta$ and $k$ will be independent of $\gamma$. That they can be satisfied follows from the fact that given any $\delta'$, any embedded copy of $\Box_{model}(\epsilon)$ in $\mathbb{H}^n$ is $\delta'$-close to some element of $\mathcal{B}_{\Sigma_k}^{\epsilon}(\gamma)$, provided $\delta$ was chosen small enough and $k$ was chosen large enough at the start.  

If $k$ was chosen large enough, then every point on $\gamma$ will be contained in some element of $\mathcal{B}_{\Sigma_k}^{\epsilon}(\gamma)$.  Suppose that $p \in \gamma$ is contained in some $\Box_p \in \mathcal{B}_{\Sigma_k}^{\epsilon}(\gamma)$, whose boundary intersects $\gamma$ at $p_1$ and $p_2$.  Then provided $\delta$ was chosen small enough and $k$ was chosen large enough, we can find some $\Box_{p_1} \in \mathcal{B}_{\Sigma_k}^{\epsilon}(\gamma)$ that contains $p_1$, and whose boundary intersects  $\gamma$ at points $p'$ and $p''$ so that 
\[
\min (d(p',p_1), d(p'',p_1) >\epsilon/4.  
\]

In this case, again provided $\delta$ was chosen small enough, $\Box_{p_1}$ and $\Box_p$ satisfy the interlocking property, and in a similar way we can find a $\Box_{p_2}$ containing $p_2$ in its interior so that $\Box_{p_2}$ and $\Box_p$ satisfy the interlocking property.  Continuing in this way, we can find a collection of $\Box \in \mathcal{B}_{\Sigma_k}^{\epsilon}(\gamma)$ that enclose $\gamma$.  This shows that $\Sigma_k$ satisfies the strong well-distribution property for large enough $k$.

\subsection{} 
Suppose that $\Sigma$ satisfies the strong well-distribution property.  Then we claim that $\Sigma_N$ is well-distributed in $N$.  First, note that each hypercube $\Box$ bounded by lifts of  $\Sigma$ to the universal cover corresponds to a hypercube  $\Box'$ in $N$ bounded by lifts of $\Sigma_N$.  This can be seen as follows.  Suppose that $\Box$ is bounded by $2n+2$ totally geodesic hyperplanes $H_1,..,H_{2n+2}$ that are lifts of $\Sigma$.  

Recall the discussion in \ref{backgroundnegcurv}. The map $F$ defines a homeomorphism between the boundary at infinity of $\mathbb{H}^n$ and the boundary at infinity of the universal cover $\tilde{N}$ of $N$.  Taking the images of the boundaries at infinity of the $H_i$ we  obtain  $2n+2$ $n-2$-discs  in $\partial_\infty \tilde{N}$.  These bound totally geodesic hyperplanes that project down to $\Sigma_N$.  We thus obtain a hypercube $\Box'$ in $\tilde{N}$ whose sides are contained in lifts of $\Sigma_N$.  

We claim that $\Sigma_N$ satisfies the well distribution property in $N$.  Let $\gamma_N$ be a geodesic in the universal cover $\tilde{N}$.  Let $\gamma$ be a geodesic in $\tilde{M}=\mathbb{H}^n$ with the same endpoints at infinity as $\gamma_N$, identifying the two boundaries at infinity as above.  Since $\Sigma$ satisfies the strong well-distribution property, we can find a sequence of hypercubes $\Box_i: i \in \mathbb{Z}$ in $\mathbb{H}^n$ whose hyperfaces are contained in lifts of $\Sigma$ so that the following are true: 

\begin{enumerate}  
	\item Exactly two hyperfaces $H_{i,top}$ and $H_{i,bottom}$ of $\Box_i$ intersect $\gamma$
	
	\item $H_{i-1,top}$ is contained in the region between $H_{i,top}$ and $H_{i,bottom}$
	
	\item Each point on $\gamma$ is contained in some $\Box_i$.  
	
	\end{enumerate} 

The hyperplanes corresponding to the faces of  $\Box_i$ excluding $H_{i,top}$ and $H_{i,bottom}$ bound an infinite solid rectangular ``pillar" that contains $\gamma$.

\begin{figure}
	\includegraphics[width=50mm]{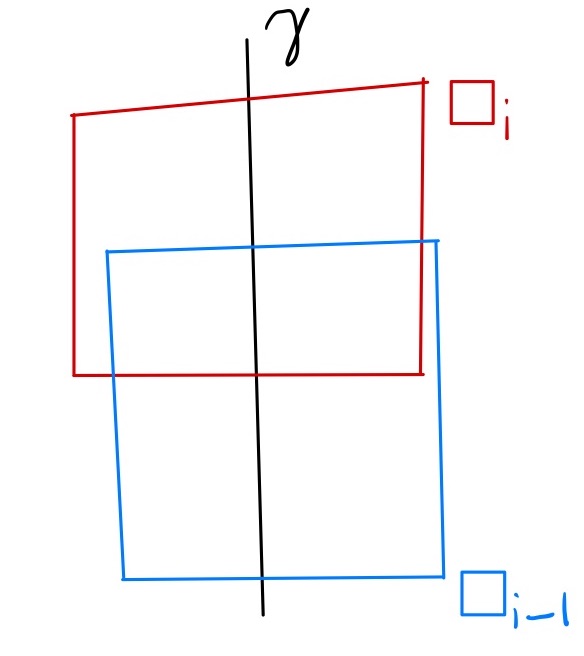}
	\caption{Hypercubes enclosing $\gamma$ } 
\end{figure}

The hypercubes $\Box_i'$ corresponding to the $\Box_i$ satisfy the same three properties in $\tilde{N}$, but with $\gamma$ replaced by $\gamma_N$.  Item (1) follows from the fact that a totally geodesic hyperplane $H$ in $\tilde{N}$ will intersect a given geodesic $\gamma$ exactly if the two endpoints of $\gamma$ in $\partial_{\infty} \tilde{N}$ are in separate components of the complement of $\partial_{\infty} H$ in $\partial_{\infty} \tilde{N}$.  This implies that whether $\gamma$ intersects $H$ is determined by what happens in $\mathbb{H}^n$ for the corresponding geodesic and lift of $\Sigma$.  The second property follows in a similar way. 

The first two properties imply that the subset of $\gamma$ of points contained in some $\Box_i'$ is both open and closed, and so must be all of $\gamma$.

Therefore, since every point in $N$ is contained in some geodesic,  every point $p$ of $\tilde{N}$ is contained in an embedded hypercube $\Box_p'$ as above.  By Theorem \ref{main}, this implies that $N$ has constant curvature $-1$.  Since $N$ is homotopy equivalent to $M$, $N$ must be isometric to $M$ by Mostow's rigidity theorem.

\section{Examples Where Rigidity Fails} \label{counterexample} 

In this section we prove the following theorem.  We expect that a similar construction works in higher dimensions, but for simplicity we focus on the three dimensional case.

\begin{thm} \label{ewrf} 
There exists a closed hyperbolic 3-manifold $(M,g_{hyp})$ and a metric $g$ on $M$  that is not isometric to $g_{hyp}$ such that the following hold
	\begin{enumerate}
	\item $(M,g_{hyp})$  contains a closed totally geodesic surface $S$.  
	\item $S$ is a totally geodesic hyperbolic surface in the metric $(M,g)$ 
	\item The sectional curvature of $g$ is less than or equal to $-1$.
	\end{enumerate}
	\end{thm}

Our construction will be based on the following proposition.    

\begin{prop} \label{prop1}
Let $\ell>1$ and let  $g_\ell$ be a metric on $S\times \mathbb{R}$ with constant curvature $-\ell$ so that $S\times \{0\}$ is a  totally geodesic surface in the metric $g_\ell$.  There exists a smooth metric $g$ on $S\times \mathbb{R}$ satisfying the following conditions: 

\begin{enumerate} 
	\item $g=g_\ell$ outside of a compact set 
	\item $g$ has sectional curvature $K_g \leq -1$ 
	\item $S \times \{0\}$ is totally geodesic in $g$ with constant sectional curvature $-1$.  
	\end{enumerate}
\end{prop}

V. Lima constructed non-hyperbolic examples as in the proposition, but with point (1) excluded.  The examples we give were explained to us by Laurent Mazet, and we are grateful to him for allowing us to include them.  Our verifications closely follow \cite{l19}[section 3].

\begin{proof}
   Let $(S,g_{hyp})$ be a hyperbolic surface, and let $f$ be a function defined by 
\[
f(t):= \frac{1}{\ell }\cosh (\ell t ) + \chi(t)(1- \frac{1}{\ell})
\]
for $\chi(t)$ a smooth non-negative function satisfying, for $\epsilon>0$ and $M$ respectively sufficiently small and large to be specified later: 

\begin{enumerate}
	\item  $\chi(t)=1$ near $t=0$, $\chi(t)=0$ for $|t|>M$ 
	\item $\max(|\chi'(t)|, |\chi''(t)|)< \epsilon$
	\end{enumerate}
 
 We take $\ell>1$.  Define $g$ to be  the warped product metric 
\[
g:= f^2(t) g_{hyp}^2 + dt^2.   
\]

Writing $g_\ell$ as 
\[
g_\ell = \frac{1}{\ell^2 }\cosh^2 (\ell t ) g_{hyp}^2 + dt^2, 
\]

Only the second item in the statement of the proposition  requires an argument.  In computing the curvature of $g$, we follow \cite{l19}.  For a vector field $V$ on $(S\times \mathbb{R},g)$, define the projection to $S$ by 
\[
V^S := V - g(V,\partial_t) \partial_t 
	\]
	
	The curvature tensor $R$ of $g$ can be written in terms of the curvature tensor $R_S$ of $(S,g_{hyp})$ as 
	\begin{align}
		R(X,Y,Z)&= R_S(X^S,Y^S)Z^S - \left(\frac{f'}{f}\right)^2 (g(X,Z)Y -g(Y,Z)X) \\
		&+ (\log f)'' g(Z,\partial_t) (g(Y,\partial_t)X -g(X,\partial_t)Y)\\
		&- (\log f)''(g(X,Z)g(Y,\partial_t)- g(Y,Z)g(X,\partial_t)\partial_t,
		\end{align}
	
	for $X$, $Y$, and $Z$ smooth vector fields.  Given a plane $\Pi$ spanned by unit length orthogonal vector fields $X,Y$, we have that the sectional curvature of $\Pi$ is given by 
	\begin{align} \label{curv}
		K_{\Pi}&= \frac{K_S - (f')^2}{f^2} + \left( \frac{-K_S + (f')^2 -f'' f}{f^2}\right)(g(X,\partial_t)^2 + g(Y,\partial_t)^2) 
		\end{align}

	For $|t|$ smaller than some $t_0$ we have that $\chi(t)$ is equal to $1$.  The numerator of the coefficent of $(g(X,\partial_t)^2 + g(Y,\partial_t)^2)$ is then equal to 
		\begin{align}
	&= 1 + \sinh^2(\ell t) - \ell \cosh \ell t (\frac{1}{\ell} \cosh(\ell t) + \chi(t) (1-\frac{1}{\ell})) \\
	&=(1-\ell)\cosh \ell t, 
	\end{align}
	
 which is less than zero, so we have that (\ref{curv}) is maximized when $(g(X,\partial_t)^2 + g(Y,\partial_t)^2)=0$.  In this case, we just need to show that the initial term
	\begin{equation} \label{q1}
	 \frac{K_S - (f')^2}{f^2}= \frac{-1-(\sinh(\ell t )+\chi'(t)(1-\frac{1}{\ell}))^2}{\left(\frac{1}{\ell} \cosh(\ell t) + \chi(t)(1-\frac{1}{\ell})\right)^2} 
	\end{equation}
\noindent is at most $-1$.  This  will actually be true for all $t$, not just $|t|<t_0$, which we verify now.  Note that this holds at $t=0$.   It is then enough to show that the quantity (\ref{q1}) is decreasing in $t$ for $t>0$ and increasing in $t$ for $t<0$.  For the derivative of the  quantity (\ref{q1}) to be negative, we need 
\vspace{-2mm} 
\begin{equation} \label{expn} 
	\end{equation} 

\vspace{-8mm} 

\begin{align*} 
0>&-2 (\sinh(\ell t )+\chi'(t)(1-\frac{1}{\ell}))  (\ell \cosh(\ell t )  + \chi''(t)(1-\frac{1}{\ell})) (\frac{1}{\ell} \cosh (\ell t) + \chi(t)(1-\frac{1}{\ell})) + \\
& (1+(\sinh(\ell t )+\chi'(t)(1-\frac{1}{\ell}))^2)(\sinh(\ell t) + \chi'(t) (1-\frac{1}{\ell}))  .  
\end{align*} 

Choose $t_0$ so that $\chi(t)=1$ for $|t|<t_0$.  We first do the case  $|t|<t_0$. Assume that $t>0$.    To verify  (\ref{expn}), it is enough to show that

\begin{equation} \label{ineq3}
 0  \geq \sinh (\ell t )(-2\cosh^2(\ell t) -2\ell \cosh(\ell t ) \chi(t) (1-\frac{1}{\ell}) + 1 + \sinh^2(\ell t)), 
\end{equation} 

\noindent where the RHS of (\ref{ineq3}) is all of the terms of the RHS of (\ref{expn}) that don't contain $\chi'(t)$ or $\chi''(t)$.      Since $\sinh(\ell t)(-2\ell \cosh(\ell t ) \chi(t) (1-\frac{1}{\ell}))$ is always non-positive provided $t>0$, for (\ref{ineq3}) to hold  it is enough that 

\begin{equation}
0\geq  -2\cosh^2(\ell t ) + 1  +\sinh^2(\ell t), 
\end{equation}

\noindent which is true since the RHS equals $ - \cosh^2(\ell t)$.  Thus the quantity (\ref{q1}) is negative if $t_0>t>0$, and the same reasoning shows it is positive if $-t_0<t<0$.  

Now assume that $t\geq t_0$ (the case $t\leq -t_0$ is similar.)  Then by the previous computations, we will be able to conclude that (\ref{q1}) is smaller than $-1$ for all $t$ as long as $-\sinh(\ell t) \cosh^2(\ell t)$   is larger in absolute value than all of the terms of the RHS of (\ref{expn}) that, when expanded out, contain $\chi'(t)$ or $\chi''(t)$.  This will be true provided $\epsilon$ was taken small enough at the start, since $\max(|\chi'(t)|, |\chi''(t)|) <\epsilon$ and no term in (\ref{expn}) has more than three factors equal to $\sinh(\ell t)$ or $\cosh(\ell t)$.  We have thus shown that (\ref{q1}) is smaller than $-1$ for all $t$.

We now check the case that $|t|>t_0$.  Because the initial term in (\ref{curv}) is always less than or equal to $-1$, if the second term in (\ref{curv}) is negative then we are done.  If not, it is enough to check (\ref{curv}) in the case that $g(X,\partial_t)^2 + g(Y,\partial_t)^2$ assumes its maximum possible value:  
\[
g(X,\partial_t)^2 + g(Y,\partial_t)^2=1.  
\]
The quantity (\ref{curv}) is then equal to $-f''/f$.  Therefore we need to show that 

\[
-1\geq -f''/f = -\frac{\ell^2 \cosh(\ell t) + \ell \chi''(t)(1-\frac{1}{\ell})}{\cosh(\ell t) + \ell \chi (t) (1-\frac{1}{\ell})})). 
\]

Since $\chi \leq 1$, it is enough to show that
\begin{equation} \label{eqn22}
1 \leq \frac{\ell^2 \cosh(\ell t)}{\cosh(\ell t) + \ell-1} + \frac{\chi''(t)(\ell-1)}{\cosh(\ell t) + \ell-1}.  
\end{equation}

Note that 
\[
\ell^2 \cosh(\ell t) \geq \ell \cosh(\ell t) + \ell(\ell-1), 
\]

so the first term on the RHS of (\ref{eqn22}) is greater than or equal to 
\[
\frac{\ell \cosh(\ell t) + \ell(\ell-1)}{\cosh(\ell t) + \ell-1} = \ell.  
\]

We can make the second term on the RHS of (\ref{eqn22}) as small as desired by choosing $\epsilon$ small enough because $ |\chi''|<\epsilon$. Therefore inequality (\ref{eqn22}) holds, which completes the proof of Proposition \ref{prop1}.  

\end{proof}

We also need the following lemma.  

\begin{lem} \label{norminj} 
Let $M$ be a closed hyperbolic 3-manifold containing an embedded two-sided totally geodesic surface $S$.   Then for any $L>0$, there exists a finite cover $M_0$ of $M$ which contains a totally geodesic surface $S_0$ isometric to $S$ and with normal injectivity radius greater than $L$.  
\end{lem}

Recall that $S_0$ has normal injectivity radius greater than $L$ in $M_0$ if no two distinct geodesic segments of length less than $L$ beginning normal to $S_0$ intersect.  This means that the normal exponential map $S_0 \times (-L,L)\rightarrow M_0$  is a diffeomorphism onto its image, for $S_0 \times (-L,L)$ identified with a tubular neighborhood of the zero section of the normal bundle of $S_0$ in $M_0$ in the natural way.   

\begin{proof}
This is a consequence of subgroup separability for $\pi_1(S,p) \subset \pi_1(M,p)$, for some choice of basepoint $p$ on $S$ \cite{mr03}[Lemma 5.3.6]. This means that for any finite subset  $\{g_1,..g_n\} \subset\pi_1(M,p)$ disjoint from $\pi_1(S,p)$, we can choose a finite index subgroup of $\pi_1(M,p)$ containing $\pi_1(S,p)$ but none of $g_1,..g_n$.

  Let $F$ be the Fuchsian cover of $M$ corresponding to $\pi_1(S)$.  Fix a connected polyhedral fundamental domain $P$ for the action of $\pi_1(M)$ on the universal cover $\mathbb{H}^3$ of $M$. Then a fundamental domain in $\mathbb{H}^3$ for the cover $F$ is tiled by copies of $P$, and any fixed normal neighborhood of $S$ in $F$ is contained in the projections of finitely many such copies.    The images of the projections to $F$ of the copies of $P$ correspond to cosets of $\pi_1(S)$ in $\pi_1(M)$.  Choose finitely many coset representatives $g_1,..,g_n$ so that the corresponding projections of copies of $P$ contain the $L$-neighborhood of $S$ in $F$.  

Now use subgroup separability to choose a finite index subgroup of $\pi_1(M)$ containing $\pi_1(S)$ but none of $g_1,..,g_n$.  The finite cover of $M$ corresponding to this finite index subgroup is then also covered by $F$, and the restriction of the covering map to the $L$-neighborhood of $S$ in $F$ is a diffeomorphism onto its image.  This shows that we can make the normal injectivity radius of $S$ as large as desired by passing to finite covers.  
\end{proof}

We can now prove Theorem \ref{ewrf}. (Compare \cite{l21}[Section 6], where a similar  construction appeared.) Take a closed hyperbolic 3-manifold $(M,g_{hyp})$ containing an embedded two-sided totally geodesic surface $S$.  See \cite{mr03} for  examples.

 Let $F\cong S \times \mathbb{R}$ be the Fuchsian cover of $(M,g_{hyp})$ corresponding to $S$.   Let  $g_\ell$ be the homothetic scaling of the hyperbolic metric on $F$ that has constant curvature $-\ell$ for some $\ell>1$.  We apply Proposition \ref{prop1} to obtain a metric $g_K$ on $F$ that agrees with $g_\ell$ outside a compact set $K$, and so that $S \times \{0\}$ is a totally geodesic hyperbolic surface.  

Next we apply Lemma \ref{norminj} to find a finite cover of $M'$ so that $F$ also covers $M'$, and the restriction of the covering map $\rho:F \rightarrow M'$ to $K$ is injective.  We can then define a metric on $M'$ to be the scaled hyperbolic metric with constant curvature $-\ell$ outside of $\rho(K)$, and to be the pushforward of $g_K$ under $\rho$ on $\rho(K)$.  Since $g_\ell$ and $g_K$ agree outside of the compact set $K$, this gives a smooth well-defined metric on $M'$, which completes the proof.

\bibliographystyle{amsalpha} 
\bibliography{bibliography}
\end{document}